\documentclass{amsart}

\newtheorem{thm}{Theorem}[section]
\newtheorem{cor}[thm]{Corollary}
\newtheorem{lem}[thm]{Lemma}

\theoremstyle{definition}
\newtheorem{defn}[thm]{Definition}
\theoremstyle{remark}
\newtheorem*{rem}{Remark}
\numberwithin{equation}{section}
\usepackage{amssymb}
\usepackage{amsthm}
\usepackage{amsmath}
\usepackage{amsfonts}
\usepackage[mathscr]{eucal}
\usepackage{enumerate}
\usepackage{hyperref}

\begin{document}

%\matkboth{E.~Gselmann}{Entropy functions anf functional equations}

\title{Entropy functions and functional equations}

\author[E.~Gselmann]{Eszter Gselmann}

\address{Institute of Mathematics\\
University of Debrecen\\
P. O. Box: 12.\\
Debrecen\\
H--4010\\
Hungary}

\email{gselmann@science.unideb.hu}

\begin{abstract}
The purpose of this note is to give the general solution of 
two functional equations connected to the Shannon entropy and also 
to the Tsallis entropy. As a result of this, we present the regular solution of 
these equations, as well. Furthermore, we point out that 
the regularity assumptions used in previous works can substantially be weakened. 
\end{abstract}



\maketitle

\section{Introduction and preliminaries}

Since the celebrated paper of Claude E.~Shannon (see \cite{Sha48}) appeared, 
the information theory has become an extensive branch of mathematics. 
Furthermore, it is known that information measures can be characterized 
via functional equations. 
Concerning this, the reader can consult the two basic monographs 
Acz\'{e}l--Dar\'{o}czy \cite{AD75} and Ebanks--Sahoo--Sander \cite{ESS98}. 

Although the characterization problem of information measures 
nearly comes to the end, from time to time one can meet new functional 
equations from this area. 
A possible explanation for this is that the Shannon entropy and also the 
entropy of degree alpha (or Tsallis entropy) has been re--discovered by 
physicists and engineers, see Dar\'{o}czy \cite{Dar70} and Tsallis \cite{Tsa88}. 

The aim of this note is to give the general solution of two functional 
equations connected to the notion of the Shannon entropy and also that of the 
Tsallis entropy. More precisely, 
in the second section we will firstly solve the equation 
\[
 \tag{$\ast$} 
f(xy)+f((1-x)y)-f(y)=\left(f(x)+f(1-x)\right)y^{q}, 
\]
which is supposed to hold for the unknown function 
$f:]0, 1]\rightarrow\mathbb{R}$ for all $x\in ]0, 1[$ and 
$y\in ]0, 1]$, where $q\in\mathbb{R}$ is a fixed parameter. 

For the unknown function $f:]0, 1]\rightarrow\mathbb{R}$ the equation
\[
 \tag{$\ast\ast$}
f(xy)=\left(\frac{x^{\alpha}+x^{\beta}}{2}\right)f(y)+\left(\frac{y^{\alpha}+y^{\beta}}{2}\right)f(x)
\]
will also be solved which is assumed to hold for all $x, y\in ]0, 1]$, where 
$\alpha, \beta\in\mathbb{R}$ are fixed parameters. 

These two functional equations were solved in Sharma--Taneja \cite{ST75} and also 
in Furuichi \cite{Fur10} under the assumptions that the unknown function is 
\emph{nonnegative} and \emph{differentiable} and they called the solutions of 
these equations \emph{entropy functions} -- to this alludes the title of the present work. 
First we give the general solution 
of these equations and then we will point out that the regularity suppositions 
(that is, nonnegativity and differentiability) can \emph{essentially} be weakened 
to get the same result as that of \cite{Fur10, ST75}. 

In what follows some preliminary definitions and results will be listed, mainly 
from the theory of functional equation, these results can also be found in 
Kuczma \cite{Kuc85}. 

\begin{defn}\label{D1.1}
 Let $I\subset\mathbb{R}$ and
$
\mathscr{A}=\left\{(x, y) \vert x, y, x+y\in I\right\}.
$
A function
$a:I\rightarrow\mathbb{R}$ is called \emph{additive on $\mathscr{A}$} if
\begin{equation}\label{Eq1.2}
a\left(x+y\right)=a\left(x\right)+a\left(y\right)
\end{equation}
holds for all pairs $(x, y)\in \mathscr{A}$. \\
\noindent
Consider the set
$
\mathscr{I}=\left\{(x, y) \vert x, y, xy \in I\right\}.
$
We say that
$\mu:I\rightarrow\mathbb{R}$ is \emph{multiplicative on $\mathscr{I}$} if the functional equation
\begin{equation}\label{Eq1.3}
\mu\left(x y\right)=\mu\left(x\right)\mu\left(y\right)
\end{equation}
is fulfilled for all $(x, y)\in\mathscr{I}$. \\
\noindent
A function
$\ell:I\rightarrow\mathbb{R}$ is called \emph{logarithmic on $\mathscr{I}$}
if it satisfies the functional equation
\begin{equation}\label{Eq1.4}
\ell\left(x y\right)=\ell\left(x\right)+\ell\left(y\right)
\end{equation}
for all $(x, y)\in \mathscr{I}$.
\end{defn} 

Henceforth, for all $n\geq 2$ we define the set $D_{n}$ by
\[
 D_{n}=\left\{(x_{1}, \ldots, x_{n})\in\mathbb{R}^{n} 
\vert x_{1}, \ldots, x_{n}, \sum_{i=1}^{n}x_{i}\in ]0, 1[\right\}. 
\]

As we wrote above, we will also determine the regular solutions of equations 
$(\ast)$ and $(\ast\ast)$. 
To do this, the following regularity theorems 
will be applied. 

\begin{lem}\label{L1.1}
 Let $a:]0, 1[\rightarrow\mathbb{R}$ be an additive function on the set $D_{2}$ and assume that 
\begin{enumerate}[(i)]
 \item $a$ is bounded above or below on subset of $]0, 1[$ that has positive 
Lebesgue measure;
\item or $a$ is Lebesgue measurable. 
\end{enumerate}
Then there exists $c\in\mathbb{R}$ such that 
\[
 a(x)=cx
\]
holds for all $x\in ]0, 1[$. 
\end{lem}

\begin{lem}\label{L1.2}
 Let $\ell:]0, 1[\rightarrow\mathbb{R}$ be a logarithmic function on the set 
\[
\widetilde{\mathscr{L}}=\left\{(x, y) \vert x, y, xy\in ]0, 1[\right\}
\]
and assume that 
\begin{enumerate}[(i)]
 \item $\ell$ is bounded above or below on subset of $]0, 1[$ that has positive 
Lebesgue measure;
\item or $\ell$ is Lebesgue measurable. 
\end{enumerate}
Then there exists $c\in\mathbb{R}$ such that 
\[
 \ell(x)=\ln(x)
\]
holds for all $x\in ]0, 1[$. 
\end{lem}

We also mention that in case $a:]0, 1[\rightarrow\mathbb{R}$ is an additive 
function on the set $D_{2}$, then it can be uniquely extended to function 
$\widetilde{a}:\mathbb{R}\rightarrow\mathbb{R}$ which is additive on 
$\mathbb{R}$ (cf. Kuczma \cite{Kuc85}). 
For the sake of simplicity we will always bear in mind this fact, and 
the extension of the function in question will always be denoted by 
the same character. 

The notion of derivations will also be utilized in the next section, see 
Kuczma \cite{Kuc85}.

\begin{defn}
 An additive function $a:\mathbb{R}\rightarrow\mathbb{R}$ is termed to be 
a \emph{real derivation}, if it also fulfills the equation 
\[
 d(xy)=xd(y)+yd(x)
\]
for all $x, y\in\mathbb{R}$. 
\end{defn}
 
From this definition immediately follows that every real derivation vanishes at the 
rationals. 
Additionally, something more is true. 
Namely, every real derivation is identically zero on the set of
algebraic numbers (over the rationals). 
Furthermore, if a real derivation is Lebesgue measurable or 
bounded above or below on the set that has positive 
Lebesgue measure, then it is identically zero. 
Therefore, it can be seen that the non--trivial real derivations can be very irregular. 
Although it is surprising, there exists non identically zero real derivation, see
Theorem 14.~2.~2 in Kuczma \cite{Kuc85}. 

The following lemma was proved in \cite{GM09}. 

\begin{lem}\label{L1.3}
 Suppose that the function 
$\varphi:]0, +\infty[\rightarrow\mathbb{R}$ is such that 
\[
 \varphi(xy)=x\varphi(y)+y\varphi(x) 
\qquad 
\left(x, y\in ]0, 1[\right)
\]
and the function $g:]0, 1[\rightarrow\mathbb{R}$ defined by 
\[
 g(x)=\varphi(x)+\varphi(1-x) 
\qquad
\left(x\in ]0, 1[\right)
\]
is Lebesgue measurable or it is bounded (above and below) on a subset of 
$]0, 1[$ that has positive Lebesgue measure. 
Then there exist $c\in\mathbb{R}$ and a real derivation 
$d:\mathbb{R}\rightarrow\mathbb{R}$ such that 
\[
 \varphi(x)=cx\ln(x)+d(x)
\]
is fulfilled for any $x\in ]0, 1[$. 
\end{lem}

During the proof of our main theorem concerning equation $\left(\ast\right)$ 
we will apply a result concerning the 
so--called cocycle equation, see Jessen--Karpf--Thorup \cite{JKT68}. 
In the proof however this equation (i.e., the cocycle equation) will not be satisfied on the whole domain 
but only on a restricted one. 
Therefore we will apply a result of Ng \cite{Ng74}, in which the author solves the cocycle
equation on a restricted domain, 
see also Acz\'{e}l--Ng \cite{AN83} and 
Ebanks--Sahoo--Sander \cite{ESS98}.

\begin{thm}[Ng \cite{Ng74}]\label{T1.1}
 Let $\mu:]0, 1[\rightarrow\mathbb{R}$ be a given multiplicative 
function and $G:D_{2}\rightarrow\mathbb{R}$ be a function. 
Then the general solution of the system of functional 
equations 
\begin{equation}\label{Eq1.5}
G(x, y)=G(y, x);
\qquad 
\left((x, y)\in D_{2}\right) 
\end{equation}
\begin{equation}\label{Eq1.6}
 G(x, y)+G(x+y, z)=G(y, z)+G(x, y+z);
\qquad 
\left((x, y, z)\in D_{3}\right)
\end{equation}
and 
\begin{equation}\label{Eq1.7}
G(tx, ty)=\mu(t)G(x, y) 
\qquad 
\left(t\in ]0, 1[, (x, y)\in D_{2}\right) 
\end{equation}
is given by in case $\mu(x)=x$, 
\begin{equation}
 G(x, y)=\varphi(x)+\varphi(y)-\varphi(x+y), 
\qquad 
\left((x, y)\in D_{2}\right)
\end{equation}
where $\varphi:]0, +\infty[\rightarrow\mathbb{R}$ is such that 
\[
 \varphi(xy)=y\varphi(x)+x\varphi(y),  
\qquad 
\left(x, y\in ]0, +\infty[\right)
\]
otherwise there exists $c\in\mathbb{R}$ such that 
\[
 G(x, y)=c\left[\mu(x)+\mu(y)-\mu(x+y)\right]
\]
is fulfilled for all $(x, y)\in D_{2}$. 
\end{thm}

\section{Main results} 

In this section we will find the 
general solutions of equations $\left(\ast\right)$ and 
$\left(\ast\ast\right)$. After this, the regular solutions of these 
equations will be presented. Furthermore, it will be pointed out that 
the regularity assumptions of \cite{Fur10} and \cite{ST75} can be 
substantially weakened. Moreover, in some cases these suppositions can even be 
omitted. 

\begin{thm}\label{T2.1}
 Let $q\in\mathbb{R}$ be arbitrarily fixed, then the function 
$f:]0, 1]\rightarrow\mathbb{R}$ fulfills equation 
\begin{equation}\label{Eq2.1}
 f(xy)+f((1-x)y)-f(y)=\left(f(x)+f(1-x)\right)y^{q}
\end{equation}
for all $x\in ]0, 1[$ and $y\in ]0, 1]$, in case 
$q\neq 1$, if and only if, there exist $c\in\mathbb{R}$ and an additive function $a:\mathbb{R}\rightarrow\mathbb{R}$ such that 
\[
 f(x)= 
\left\{
\begin{array}{rcl}
 a(x)+cx^{q},& \text{ if }& x\in ]0, 1[ \\
0, & \text{ if }& x=1
\end{array}
\right.
\]
furthermore, in case $q=1$, if and only if there exists an additive function $a:\mathbb{R}\rightarrow\mathbb{R}$  and a function 
$\varphi:]0, +\infty[\rightarrow\mathbb{R}$ such that 
\[
 \varphi(xy)=x\varphi(y)+y\varphi(x) 
\qquad 
\left(x, y \in ]0, +\infty[\right)
\]
and 
\[
 f(x)=
\left\{
\begin{array}{rcl}
 a(x)+\varphi(x),& \text{ if }& x\in ]0, 1[ \\
0, & \text{ if }& x=1
\end{array}
\right.
\]
is fulfilled. 
\end{thm}
\begin{proof}
Assume that the function $f:]0, 1]\rightarrow\mathbb{R}$  fulfills equation 
\eqref{Eq2.1}. With the substitution $y=1$ we immediately get that $f(1)=0$. 
Therefore it is enough to restrict ourselves to the interval $]0, 1[$. 
Let $(u, v)\in D_{2}$ and let us replace in equation \eqref{Eq2.1} 
$x$ by $\dfrac{u}{u+v}$ and $y$ by $(u+v)$, respectively. 
In this case we obtain that 
\begin{equation}\label{Eq2.2}
 f(u)+f(v)-f(u+v)=\left[f\left(\frac{u}{u+v}\right)+f\left(\frac{v}{u+v}\right)\right] (u+v)^{q}
\end{equation}
holds for all $(u, v)\in D_{2}$. 

Define the functions $\mathscr{C}_{f}$ and $\mathscr{R}_{f}$ on the set $D_{2}$ by 
\[
 \mathscr{C}_{f}(u, v)=f(u)+f(v)-f(u+v) 
\qquad 
\left((u, v)\in D_{2}\right)
\]
and 
\[
 \mathscr{R}_{f}(u, v)=
\left[f\left(\frac{u}{u+v}\right)+f\left(\frac{v}{u+v}\right)\right] (u+v)^{q}. 
\qquad 
\left((u, v)\in D_{2}\right)
\]
With this notations equation \eqref{Eq2.2} yields that 
\[
 \mathscr{C}_{f}(u, v)=\mathscr{R}_{f}(u, v). 
\qquad 
\left((u, v)\in D_{2}\right)
\]
Let us observe that the function $\mathscr{R}_{f}$ is $q$--homogeneous. 
Indeed, for all $t\in ]0, 1[$ and $(u, v)\in D_{2}$
\begin{multline*}
 \mathscr{R}_{f}(tu, tv) =
\left[f\left(\frac{tu}{tu+tv}\right)+f\left(\frac{tv}{tu+tv}\right)\right] (tu+tv)^{q}
\\
=
t^{q}\left[f\left(\frac{u}{u+v}\right)+f\left(\frac{v}{u+v}\right)\right] (u+v)^{q} =
t^{q}\mathscr{R}_{f}(u, v). 
\end{multline*}
This implies that the function $\mathscr{C}_{f}$ is also a $q$--homogeneous function. 
Furthermore, the function $\mathscr{C}_{f}$ is symmetric and also fulfills the cocycle 
equation. All in all, this means that the function $\mathscr{C}_{f}$ satisfies 
equations \eqref{Eq1.5}, \eqref{Eq1.6} and \eqref{Eq1.7} with the multiplicative function 
$\mu(t)=t^{q}$. Thus by Theorem \ref{T1.1}, in case $q\neq 1$ there exists 
$c\in\mathbb{R}$ such that 
\begin{equation}\label{Eq2.3}
 \mathscr{C}_{f}(x, y)= c\left[x^{q}+y^{q}-(x+y)^{q}\right], 
\qquad 
\left((x, y)\in D_{2}\right)
\end{equation}
and in case $q=1$ there exists a function 
$\varphi:]0, +\infty[ \rightarrow\mathbb{R}$ such that 
\[
 \varphi(xy)=x\varphi(y)+y\varphi(x)
\qquad 
\left(x, y \in ]0, +\infty[\right)
\]
and 
\begin{equation}\label{Eq2.4}
 \mathscr{C}_{f}(x, y)=\varphi(x)+\varphi(y)-\varphi(x+y) 
\qquad 
\left((x, y)\in D_{2}\right)
\end{equation}
is satisfied.

Firstly, we deal with the case $q\neq 1$. Define the function 
$\widetilde{f}:]0, 1[\rightarrow\mathbb{R}$ by 
\[
 \widetilde{f}(x)=f(x)-cx^{q}, 
\qquad 
\left(x\in ]0, 1[\right)
\]
then equation \eqref{Eq2.3} yields that the function 
$\widetilde{f}$ is additive on the set $D_{2}$. 
Regarding the function $f$ this shows that there exists an 
additive function 
$a:\mathbb{R}\rightarrow\mathbb{R}$ such that 
\[
 f(x)=cx^{q}+a(x)
\]
holds for any $x\in ]0, 1[$. Since $f(1)=0$, 
function $f$ has the form which had to be proved. 

If $q=1$, then let us define the function 
$\widetilde{f}:]0, 1[\rightarrow\mathbb{R}$ by 
\[
 \widetilde{f}(x)=f(x)-\varphi(x). 
\qquad
\left(x\in ]0, 1[\right)
\]
Equation \eqref{Eq2.4} yields that the function $\widetilde{f}$ is 
an additive function on $D_{2}$. 
Concerning the function $f$ from this we get that 
there exists an additive function $a:\mathbb{R}\rightarrow\mathbb{R}$ such that 

\[
 f(x)=\varphi(x)+a(x)
\]
is fulfilled for any $x\in ]0, 1[$, where the function $\varphi:]0, +\infty[\rightarrow\mathbb{R}$ 
satisfies
\[
 \varphi(xy)=x\varphi(y)+x\varphi(y). 
\qquad 
\left(x, y\in ]0, +\infty[\right)
\]

Since $f(1)=0$ has to hold, the function $f$  is of the form 
\[
 f(x)=
\left\{
\begin{array}{rcl}
 a(x)+\varphi(x),& \text{ if }& x\in ]0, 1[ \\
0, & \text{ if }& x=1
\end{array}
\right.
\]
that had to be proved. 
The converse direction is an easy computation. 
\end{proof}

The following corollary contains the regular solutions of equation \eqref{Eq2.1}. 

\begin{cor}\label{C2.1}
 Let $q\in\mathbb{R}$ be arbitrary and suppose that the function $f:]0, 1]\rightarrow\mathbb{R}$ 
satisfies equation \eqref{Eq2.1} for all $x\in ]0, 1[$ and $y\in ]0, 1]$. 
If $q\neq 1$ assume further that one of the following is true. 
\begin{enumerate}[(i)]
 \item $f$ is bounded above or below on a subset of $]0, 1[$ that has positive 
Lebesgue measure;
\item $f$ is Lebesgue measurable. 
\end{enumerate}
Then there exist $c, c^{\ast}\in\mathbb{R}$ such that
\[
f(x)=
\left\{
\begin{array}{rcl}
c^{\ast}x+cx^{q}, &\text{ if }& x\in ]0, 1[\\ 
0, &\text{ if }& x=1
\end{array}
\right.
\]
In case $q=1$ suppose additionally that one of the statements below hold. 
\begin{enumerate}[(i)]
 \item $f$ is bounded (above \textbf{and} below) on a subset of $]0, 1[$ that has positive 
Lebesgue measure;
\item $f$ is Lebesgue measurable. 
\end{enumerate}
Then there exist $c, c^{\ast}\in\mathbb{R}$ such that 
\[
 f(x)=
\left\{
\begin{array}{rcl}
cx\ln(x)+c^{\ast}x, & \text{ if }& x\in ]0, 1[ \\
0, &\text{ if }& x=1
\end{array}
\right.
\]
is fulfilled.  
\end{cor}

\begin{proof}
 Firstly, we investigate the case $q\neq 1$. 
From Theorem \ref{T2.1} we obtain that 
\[
 f(x)=a(x)+cx^{q}
\]
holds for all $x\in ]0, 1[$ and $f(1)=0$. 
If we rearrange this, it follows that 
\[
 a(x)=f(x)-c x^{q}. 
\qquad 
\left(x\in ]0, 1[\right)
\]
By our assumptions $f$ is bounded above or 
below on a subset of $]0, 1[$ that has positive 
Lebesgue measure, or $f$ is a Lebesgue measurable 
function. 

This implies that the additive function 
$a$ fulfills condition (i) or (ii) of 
Lemma \ref{L1.1}. From this, we obtain that there exists a 
constant $c^{\ast}\in\mathbb{R}$ such that $a(x)=c^{\ast}x$. 
This implies however that 
\[
 f(x)=c^{\ast}x+cx^{q}
\]
holds for all $x\in ]0, 1[$ and $f(1)=0$. 

Secondly, we assume that $q=1$. From Theorem \ref{T2.1}, we obtain that $f(1)=0$ and
\[
 f(x)=\varphi(x)+a(x),  
\qquad
\left(x\in ]0, 1[\right)
\]
where $\varphi:]0, +\infty[\rightarrow\mathbb{R}$ fulfills equation 
\[
 \varphi(xy)=x\varphi(y)+y\varphi(x) 
\qquad
\left(x, y \in ]0, +\infty[\right)
\]
and $a:\mathbb{R}\rightarrow\mathbb{R}$ is an additive function. 
Define the function $g:]0, 1[\rightarrow\mathbb{R}$ by 
\[
 g(x)=f(x)+f(1-x)-a(1). 
\qquad 
\left(x\in ]0, 1[\right)
\]
In this case the function $g$ is Lebesgue measurable or 
bounded on a subset of $]0, 1[$ that has positive 
Lebesgue measure. 
Furthermore, 
\[
 g(x)=f(x)+f(1-x)-a(1)=
\varphi(x)+a(x)+\varphi(1-x)+a(1-x)-a(1)=
\varphi(x)+\varphi(1-x), 
\]
where we used that $a$ is an additive function. 
All in all, this implies that the function $g$ satisfies the assumptions 
of Lemma \ref{L1.3}. Thus there exist $c\in\mathbb{R}$ and a real derivation 
$d:\mathbb{R}\rightarrow\mathbb{R}$ such that 
\[
 \varphi(x)=cx\ln(x)+d(x). 
\qquad 
\left(x\in ]0, 1[\right)
\]
Concerning the function $f$ this yields that 
\[
 f(x)=cx\ln(x)+d(x)+a(x), 
\qquad 
\left(x\in ]0, 1[\right)
\]
or if we rearrange this, 
\[
 f(x)-cx\ln(x)=d(x)+a(x). 
\qquad 
\left(x\in ]0, 1[\right)
\]
By our assumptions the function $f$ is bounded on a subset of $]0, 1[$ that has positive 
Lebesgue measure or it is Lebesgue measurable. Furthermore, the function $d(x)+a(x)$  is a 
sum of two additive function, that is, this function is also additive, which is 
bounded on a subset of $]0, 1[$ with positive Lebesgue measure 
or it is Lebesgue measurable. In view of Lemma \ref{L1.1}, 
there exists $c^{\ast}\in\mathbb{R}$ such that 
$d(x)+a(x)=c^{\ast}x$ holds for all $x\in\mathbb{R}$.  
Thus 
\[
 f(x)=cx\ln(x)+c^{\ast}x
\]
holds for all $x\in ]0, 1f$ and $f(1)=0$. 
\end{proof}

\begin{rem}
 Under the assumptions of the previous corollary, in case for the function 
$\lim_{x\to 1-}f(x)$ exists and $\lim_{x\to 1-}f(x)=f(1)$, then in case $q\neq 1$, 
\[
 f(x)=c^{\ast}(x-x^{q}) 
\qquad 
\left(]0, 1]\right)
\]
is fulfilled with some $c^{\ast}\in\mathbb{R}$. Furthermore, in case 
$q=1$, 
\[
 f(x)=cx\ln(x)
\qquad 
\left(x\in ]0, 1]\right)
\]
is satisfied with a certain $c\in\mathbb{R}$. 
\end{rem}

At this point of the paper we turn to deal with equation 
$\left(\ast\ast\right)$. Before this, we present a more general equation. 
Thus the solutions of the above mentioned functional equation 
will be showed as a corollary of the following result. 
Additionally, the regular solutions of $\left(\ast\ast\right)$ 
will be dealt with, as well. 

The following lemma was proved by E.~Vincze in 1962 
for commutative groups. Although $\left(]0, 1], \cdot\right)$ 
is not a group, only a semigroup, we remark that the method used in 
Satz 5 in Vincze \cite{Vin62} is appropriate for commutative semigroups, as well. 

\begin{lem}\label{L2.1}
 Let $g:]0, 1]\rightarrow\mathbb{R}$ be a given function, and 
$f:]0, 1]\rightarrow\mathbb{R}$ be such that 
\begin{equation}\label{Eq2.5}
 f(xy)=g(y)f(x)+g(x)f(y)
\end{equation}
holds for all $x, y\in ]0, 1]$. 

If $g$ is a multiplicative function, then 
\[
 f(x)=g(x)\ell(x), 
\qquad 
\left(x\in ]0, 1]\right)
\]
where $\ell:]0, 1]\rightarrow\mathbb{R}$ is a logarithmic function, and in case
$g$ is not a multiplicative function, that is, there exist $t_{1}, t_{2}\in ]0, 1]$ 
such that $g(t_{1}t_{2})\neq g(t_{1})g(t_{2})$, then 
\[
 f(x)=
\frac{\left[g(t_{1}x)-g(t_{1})g(x)\right]f(t_{1})}{g(t_{1}t_{2})-g(t_{1})g(t_{2})}
\]
holds for all $x\in ]0, 1]$. 
\end{lem}

\begin{cor}\label{C2.2}
 Let $\alpha, \beta\in\mathbb{R}$ be arbitrary, $f:]0, 1]\rightarrow\mathbb{R}$ be 
a function for which 
\begin{equation}\label{Eq2.9}
f(xy)=\left(\frac{x^{\alpha}+x^{\beta}}{2}\right)f(y)+\left(\frac{y^{\alpha}+y^{\beta}}{2}\right)f(x)
\end{equation}
holds for any $x, y\in ]0, 1]$. 
Then we have the following two possibilities. 
\begin{enumerate}[(i)]
 \item if $\alpha\neq\beta$, then there exist $c\in\mathbb{R}$ such that 
\[
 f(x)=c\left(x^{\alpha}-x^{\beta}\right);
\qquad 
\left(x\in ]0, 1]\right)
\]
\item if $\beta=\alpha$, then there exist a logarithmic function 
$\ell:]0, 1] \rightarrow\mathbb{R}$ such that 
\[
 f(x)=cx^{\alpha}\ell(x)
\]
holds for all $x\in ]0, 1]$. 
\end{enumerate}
\end{cor}
\begin{proof}
Firstly, let us suppose that $\alpha\neq \beta$ . 
In this case the function $g:]0, 1]\rightarrow\mathbb{R}$ 
defined by $g(x)=\frac{1}{2}\left(x^{\alpha}+x^{\beta}\right)$ is 
not a multiplicative function. Furthermore, with this notation, from 
\eqref{Eq2.9} equation \eqref{Eq2.5} follows. Thus, by Lemma \ref{L2.1}, 
\[
f(x)=
\frac{\left[g(t_{1}x)-g(t_{1})g(x)\right]f(t_{1})}{g(t_{1}t_{2})-g(t_{1})g(t_{2})} 
\]
holds for all $x\in ]0, 1]$, where $t_{1}, t_{2}\in ]0, 1]$ are arbitrarily fixed. 
After using the form of the function $g$, we get that 
\[
 f(x)=\frac{f(t_{1})}{t_{2}^{\alpha}-t_{2}^{\beta}}\left(x^{\alpha}-x^{\beta}\right), 
\qquad 
\left(x\in ]0, 1]\right)
\]
that is, there exists a constant $c\in\mathbb{R}$ such that 
\[
 f(x)=c\left(x^{\alpha}-x^{\beta}\right)
\]
is satisfied for all $x\in ]0, 1]$. 

Secondly, assume that $\alpha=\beta$. In this case the function 
$g:]0, 1]\rightarrow\mathbb{R}$ defined by $g(x)=x^{\alpha}$ is a
multiplicative function. Additionally, let us observe that with this 
notations \eqref{Eq2.9} becomes to equation \eqref{Eq2.5}. 
Therefore, due to Lemma \ref{L2.1}, there exist $c\in\mathbb{R}$ and a logarithmic function 
$\ell:]0, 1]\rightarrow\mathbb{R}$ such that 
\[
 f(x)=cx^{\alpha}\ell(x)
\]
is fulfilled for all $x\in]0, 1]$. 

Finally, a facile computation shows the correctness of the converse direction. 
\end{proof}

From this corollary we can effortlessly get the regular solutions of equation 
\eqref{Eq2.9}. Let us observe that in case $\alpha\neq \beta$, the solutions of this 
equation are regular already. Therefore in this case the 
regularity assumptions (that is, the nonnegativity and the differentiability) 
are superfluous in the papers \cite{Fur10} and \cite{ST75}. 
Furthermore, if $\alpha=\beta$, then the above mentioned regularity suppositions can be 
significantly weakened. Namely, making use of Lemma \ref{L1.2} the 
following statement can be proved. 

\begin{cor}\label{C2.3}
 Let $\alpha\in\mathbb{R}$ be arbitrarily fixed and assume that 
the function $f:]0, 1]\rightarrow\mathbb{R}$ 
fulfills equation 
\begin{equation}\label{2.10}
 f(xy)=y^{\alpha}f(x)+x^{\alpha}f(y)
\end{equation}
for all $x, y\in ]0, 1]$. 
Suppose further that one of the following statements is true. 
\begin{enumerate}[(i)]
 \item $f$ is bounded above or below on a subset of $]0, 1[$ that has positive 
Lebesgue measure;
\item $f$ is Lebesgue measurable. 
\end{enumerate}
Then there exists $c\in\mathbb{R}$ such that 
\[
 f(x)=cx^{\alpha}\ln(x). 
\qquad 
\left(x\in\mathbb{R}\right)
\]

\end{cor}

\section*{Acknowledgments}
The author is indebted to the anonymous referee. His/her suggestions improved the manuscript 
significantly.\\
This research has been supported by the Hungarian Scientific Research Fund (OTKA)
Grant NK 814 02 and by the T\'{A}MOP 4.2.1./B-09/1/KONV-2010-0007 project implemented
through the New Hungary Development Plan co-financed by the European Social Fund and
the European Regional Development Fund.

\end{document}